\documentclass[reqno, final, a4paper]{amsart}
\usepackage{color}
\usepackage[colorlinks=true,allcolors=blue
]{hyperref}
\usepackage{amsmath, amssymb, amsthm}
\usepackage{mathrsfs}
\usepackage{mathtools}
\usepackage[noabbrev,capitalize,nameinlink]{cleveref}
\crefname{equation}{}{}
\usepackage{fullpage}
\usepackage[noadjust]{cite}
\usepackage{graphics}
\usepackage{pifont}
\usepackage{cancel}
\usepackage{tikz}
\usepackage{bbm}
\usepackage[T1]{fontenc}
\usetikzlibrary{arrows.meta}

\usepackage{environ}
\usepackage{framed}
\usepackage{paralist}
\usepackage{url}
\usepackage[linesnumbered,ruled,vlined]{algorithm2e}
\usepackage[noend]{algpseudocode}
\usepackage[labelfont=bf]{caption}
\usepackage{framed}
\usepackage[framemethod=tikz]{mdframed}
\usepackage{pgfplots}
\usepackage{appendix}
\usepackage{graphicx}
\usepackage[textsize=tiny]{todonotes}
\usepackage{tcolorbox}
\usepackage{xcolor}
\usepackage[shortlabels]{enumitem}
\crefformat{enumi}{#2#1#3}
\crefrangeformat{enumi}{#3#1#4 to~#5#2#6}
\crefmultiformat{enumi}{#2#1#3}
{ and~#2#1#3}{, #2#1#3}{ and~#2#1#3}
\allowdisplaybreaks

\let\originalleft\left
\let\originalright\right
\renewcommand{\left}{\mathopen{}\mathclose\bgroup\originalleft}
\renewcommand{\right}{\aftergroup\egroup\originalright}

\crefname{algocf}{Algorithm}{Algorithms}

\crefname{equation}{}{} 
\AtBeginEnvironment{appendices}{\crefalias{section}{appendix}} 

\usepackage[color]{showkeys} 

\colorlet{refkey}{orange!20}
\colorlet{labelkey}{blue!30}

\crefname{algocf}{Algorithm}{Algorithms}

\numberwithin{equation}{section}
\newtheorem{theorem}{Theorem}[section]

\newtheorem{claim}[theorem]{Claim}
\crefname{claim}{Claim}{Claims}

\newtheorem*{question*}{Question}

\crefname{fact}{Fact}{Facts}

\theoremstyle{definition}

\newtheorem*{definition*}{Definition}

\theoremstyle{remark}

\newtheorem*{remark*}{Remark}

\newcommand{\eps}{\varepsilon}

\newcommand*{\claimproofname}{Proof of claim}
\newenvironment{claimproof}[1][\claimproofname]{\begin{proof}[#1]}{\end{proof}}

\usepackage{lineno}
\theoremstyle{plain}
\newtheorem{thm}{Theorem}[section]

\newtheorem*{clm*}{Claim}

\theoremstyle{definition}

\newtheorem{conj}[thm]{Conjecture}

\usetikzlibrary{shapes.geometric}
\usepackage{tikz,xcolor,verbatim,hyperref}

\usepackage{euflag}

\title[A note on finding large  transversals efficiently]{A note on finding large  transversals efficiently}
 \author{Michael Anastos}
 \author{Patrick Morris}

 \address{MA: Institute of Science and Technology Austria (ISTA). \newline \indent  PM: Universitat Polit\`ecnica de Catalunya (UPC), Barcelona, Spain.}

\thanks{ MA was supported by the Austrian Science Fund (FWF) [10.55776/ESP3863424] and by the European Union’s Horizon
2020 research and innovation programme under the Marie Skłodowska-Curie grant - project number 101034413 \euflag.  PM was supported  by  the European Union's Horizon Europe   Marie Sk{\l}odowska-Curie grant RAND-COMB-DESIGN - project number
101106032 {\euflag}.}

\email{michael.anastos@ist.ac.at, pmorrismaths@gmail.com}

\date{\today}

\begin{document}

\begin{abstract}
In an $n \times n$ array filled with symbols, a \emph{ transversal}  is a collection of entries with distinct rows,  columns and symbols. In this note we 
show that if no symbol appears more than $\beta n$ times, the array contains a  transversal of size $(1-\beta/4-o(1))n$. In particular, if the array  is filled with $n$ symbols, each appearing $n$ times (an equi-$n$ square), we get   transversals of size $(3/4-o(1))n$. 
Moreover, our proof gives a deterministic algorithm with polynomial running time, that finds  these  transversals. 
 
\end{abstract}

\maketitle

\vspace{-6mm}
\section{Introduction}

A  \emph{Latin square}  of order $n$ is an $n \times n$ array filled with $n$ symbols such that each row and each column contains each symbol exactly once. These objects are some of the most central in combinatorial design theory and their study dates back to Euler in the 1700s. 
A \emph{transversal}\footnote{This is sometimes referred to as a \emph{partial} transversal with the term transversal being reserved for partial transversals of size $n$, which we will refer to as \emph{full} transversals here.} in a Latin square (or more generally in any $n \times n$ array filled with symbols) is a collection of  entries of the square with no repeated row, column or symbol.  The following recent result of Montgomery \cite{montgomery2023proof} proves the existence of very large  transversals in Latin squares.

\begin{thm}[Montgomery \cite{montgomery2023proof}] \label{thm:mont}
There exists $n_0 \in \mathbb{N}$ such that every Latin square of order $n\geq n_0$ contains a  transversal of size $n-1$.
\end{thm}

Theorem \ref{thm:mont} was a huge breakthrough, improving on many previous results exploring large  transversals in Latin squares (see \cite{montgomery2023proof} and the references therein) and settling a conjecture of Brualdi \cite{brualdi1991combinatorial}.  
Some years before Brualdi, in 1967 Ryser  \cite{ryser1967neuere}  conjectured that in the case of $n$  odd 
one can actually get a  transversal of size $n$ (a \emph{full} transveral). This  remains open. The two statements together, that there are  transversals of size $n$ when $n$ is odd and $n-1$ when $n$ is even, is what has become known as the \emph{Ryser-Brualdi-Stein conjecture} with Stein's name appearing as he offered several variations and strengthenings of the conjecture \cite{stein1975transversals}. 

One of Stein's conjectures concerned \emph{equi-$n$} squares - $n \times n$ arrays filled with $n$ symbols each appearing exactly $n$ times. Note that Latin squares are particular instances of equi-$n$ squares, where additionally one bounds the number of appearances in each column and row. Stein conjectured that the row and column condition is in fact superfluous and that already being an equi-$n$ square is enough to have a large  transversal. 

\begin{conj}[Stein's conjecture \cite{stein1975transversals}] \label{conj:stein}
Every equi-$n$ square has a  transversal of size $n-1$. 
\end{conj}

It turns out that Conjecture \ref{conj:stein} was too bold, although it took  43 years to find a counterexample. This was done by Pokrovskiy and Sudakov \cite{pokrovskiy2019counterexample} who gave a nice construction refuting the conjecture.

\begin{thm}[Pokrovskiy-Sudakov \cite{pokrovskiy2019counterexample}]
There exists $n_0 \in \mathbb{N}$ such that for every integer $n\geq n_0$ there exists an equi-$n$ square with no  transversal of size larger than $n-\tfrac{1}{42}\log n$.
\end{thm}

Our work here is motivated by an attempt to `save' Stein's conjecture by showing that equi-$n$ squares still contain large  transversals, if not as large as  in Conjecture \ref{conj:stein}. This route was started already by Stein \cite{stein1975transversals} who showed that every equi-$n$ square contains a  transversal of size $(1-1/e-o(1))n\approx 0.632 n$. This was improved by Aharoni, Berger, Kotlar and Ziv \cite{aharoni2017conjecture} who used topological methods to show the existence of  transversals of size $2n/3$. Our first result is the following.

\begin{thm} \label{thm:main2}
For every $\eps>0$, there exists $n_0\in \mathbb{N}$ such that every equi-$n$ square with $n\geq n_0$ has a  transversal of size $(3/4-\eps)n$. 
\end{thm}

Moving away from equi-$n$ squares, Erd{\H{o}}s and Spencer \cite{erdos1991lopsided} considered $n\times n$ arrays which are \emph{globally bounded}. They showed that if an   $n \times n$ array is filled with symbols such that no symbol appears more than $\beta n$ times (we will call these  $\beta$-bounded $n$-squares) and $\beta<1/16$, then there is a full transversal. This influential result introduced the lopsided local lemma and has since been slightly improved to $\beta\leq 27/256$ by Bissacot, Fern\'andez, Procacci and Scoppola \cite{bissacot2011improvement} using ideas from statistical mechanics. Since these results, and the seminal work of Moser and Tardos \cite{moser2010constructive} in making the local lemma constructive, finding large  transversals is considered a benchmark application of the  performance of algorithmic versions of the local lemma. For example, Harris and Srinivasan \cite{harris2017algorithmic} and Harris \cite{harris2020new} tested their (randomised) algorithms on the problem of efficiently finding large (but not full)  transversals in $\beta$-bounded $n$-squares. 
Thus there is  a general interest in constructing  algorithms finding such  transversals.

Our main result is to give a simple \emph{deterministic} algorithm for  producing a large  transversal.    

 \begin{thm} \label{thm:main}
For every  $\eps>0$ and $0<\beta\leq 1$, there exists $n_0\in \mathbb{N}$ and
a deterministic polynomial time algorithm  that finds a  transversal of size $(1-\beta/4-\eps)n$ in any $\beta$-bounded $n$-square with $n\geq n_0$. 
\end{thm}

\cref{thm:main2} follows from \cref{thm:main} by taking $\beta=1$. Our bound of $(1-\beta/4-o(1))n$ improves on all the previous bounds \cite{harris2017algorithmic,harris2020new} coming from local lemma algorithms whenever $\beta\geq 0.14$.

\subsection*{Additional note} Whilst preparing this manuscript we learned of independent breakthrough work of Chakraborti, Christoph, Hunter, Montgomery and 
Petrov \cite{ugh} who show that in fact one can guarantee  transversals of size $n-o(n)$ in an equi-$n$-square. Although Theorem \ref{thm:main} does not achieve such large  transversals, our proof is short, simple and completely independent from theirs which uses probabilistic methods. Therefore we believe it to be of independent interest.

\subsection*{Acknowledgements} \label{sec:acknow}
We are very grateful to Matthew Kwan and Alp M\"uyesser  with whom we had many interesting  discussions leading to the results of this note. We also thank the anonymous reviewers for their suggestions  improving the presentation of this note.

\section{Proof}
For a $\beta$-bounded $n$-square $A$, we define a \emph{permutation} to be a collection of $n$ cells whose support is  the support of a permutation matrix.  We let $diag(A)$ be the permutation of $A$ corresponding to the main diagonal of $A$. For $I\subseteq[n]$ we let $A_I$ be the submatrix of $A$ that covers the entries $A_{i,j}$ with $i,j\in I$ of $A$. Given $I\subseteq [n]$ and a permutation $S_I$ of $A_I$ we define $Extend(S_I)$ to be the permutation of $A$ with cells $S_I \cup diag(A_{[n]\setminus I})$. Thus $Extend(S_I)$ is the permutation of $A$ that coincides with $S_I$ inside $A_I$ and with $diag(A)$ outside $A_I$. We say that an element belongs to or appears in some permutation $S_I$ if it belongs to or appears in some cell of $S_I$.

We also define the following properties.
\begin{itemize}
    \item $(I,S_I)$ has the Property (P1) (with respect to the matrix $A$) if every element that appears on the diagonal of $A$ either appears in $S_I$  or on the diagonal of $A_{[n]\setminus I}$ (equivalently if every element that appears in $diag(A)$ also appears in $Extend(S_I)$). 
    \item $(I,S_I)$ has the Property (P2) (with respect to the matrix $A$) if $S_I$ contains a symbol that does not appear on the diagonal of $A$.
\end{itemize}
Observe that if  there is a pair $(I,S_I)$ that has both properties (P1) and (P2),  then $Extend(S_I)$ is a permutation  that contains a larger  transversal than $diag(A)$.
Indeed $Extend(S_I)$ contains all the symbols appearing on the diagonal of $A$ (by (P1)) and at least one extra symbol (by (P2)). To prove \cref{thm:main} we suppose we have some  transversal contained in $diag(A)$ and try to find a larger  transversal.
For that, we iteratively build up pairs $(I,S_I)$ that have the Property (P1), eventually getting such a pair that also has the Property (P2), at which point we have found a larger  transversal. For the initial building blocks we use pairs $(I,S_I)$ where $I=\{i,j\}$ and $S_{I}=(A_{i,j}, A_{j,i})$ for $i,j\in [n]$ such that $A_{i,i}$ and $A_{j,j}$ are either distinct and appear multiple times in $diag(A)$ or are the same and appear sufficiently many times in $diag(A)$.  We now proceed to the proof of \cref{thm:main}

\begin{proof}[Proof of \cref{thm:main}]
Let  $A$ be a $\beta$-bounded $n$-square. Also let  $\tau:=12/(\eps \beta)$,  $C:=4^{\tau}$. 
It suffices to show that if the largest  transversal contained in $diag(A)$ has size less than $(1-\beta/4-\eps)n$, then we can find a larger  transversal in $A$ in $O(n^2)$ time. Indeed starting with $T$ being the largest  transversal contained in $diag(A)$, we can apply this  $O(n)$ times, each time updating with a larger $T$ and permuting rows and columns so that $T$ lies on the diagonal, until $T$ has the required size.  

Let $\Omega$ be the set of symbols appearing in $A$. For a subset of symbols $\chi \subseteq \Omega$ we let $supp(\chi)$ be the set of indices $i$ such that $A_{i,i}\in \chi$.   We let $\Omega_0\subseteq \Omega$ be the set of symbols $\omega$ with $|supp(\omega)|\geq 2$, that is,  the set of symbols that appear multiple times in $diag(A)$.  For $i\in [n]$, if $|supp(A_{i,i})|\leq C$  we let $R(i):=supp(A_{i,i})$, else we let $R(i):=\{i\}$. For $I\subseteq[n]$
we let $R(I):=\cup_{i\in I} R(i)$ with the convention that $R(\emptyset)=\emptyset$. Observe that $I\subseteq R(I)$ for $I\subseteq [n]$.

We initiate by setting   $I(\omega):= \emptyset $ and $S_{I(\omega)}:=\emptyset$ for each $\omega\in \Omega_0$.  Until we find a larger transversal than the largest transversal contained in $diag(A)$ (at which point we terminate the algorithm), we will iteratively construct disjoint sets of symbols $\Omega_1,\Omega_2,...,\Omega_{\tau}\subseteq \Omega\setminus \Omega_0$ as follows. For $1\leq t \leq \tau$, assuming that we have not found a larger transversal during the first $t-1$ iterations, at the beginning of iteration $t$, we are given disjoint sets of symbols  $\Omega_0,\Omega_1,...,\Omega_{t-1}$ all of which appear in $diag(A)$ and  such that for each $\omega\in \Omega^{<t}:=\cup_{i=0}^{t-1}\Omega_i$, we have an index set $I(\omega)$ and a permutation $S_{I(\omega)}$ of $A_{I(\omega)}$. Moreover for $\omega\in \Omega^{<t}$ which do \emph{not} appear in $\Omega_0$, we have that 
\begin{itemize}
    \item[(A1)]  $\omega$ belongs to some cell of $S_{I(\omega)}$,
    \item[(A2)] the pair $(I(\omega),S_{I(\omega)})$ has the Property (P1),
    \item[(A3)]  $supp(\omega)\not\subseteq I(\omega)$ (note that $|supp(\omega)|= 1$ as $\omega\notin \Omega_0$ and $\omega$ appears in $diag(A)$), and
    \item[(A4)] $|I(\omega)|\leq 4^{t-1}$ and $I(\omega) \subseteq supp(\Omega^{<{t-1}})$.
\end{itemize}
Condition (A2) implies that $Extend(S_{I(\omega)})$ contains all the symbols appearing in $diag(A)$. Moreover, for $\omega\in \Omega^{<t}\setminus \Omega_0$, since $Extend(S_{I(\omega)})$ contains the cells in $S_{I(\omega)}$,  $Extend(S_{I(\omega)})$ contains a copy of $\omega$ that appears in the submatrix $A_{I(\omega)}$ of $A$, by (A1), and a second copy of $\omega$ that appears on the diagonal of $A$ and outside $A_{I(\omega)}$, by (A3).  Finally,  (A4) controls the size and the ``location'' of $I(\omega)$ and in extension of both $A_{I(\omega)}$ and $S_{I(\omega)}$.

Given the above, we construct the set of symbols $\Omega_t$ along with the index sets $I(\omega)$ and the permutations $S_{I(\omega)}$, so that for $\omega\in \Omega_t$, (A1)-(A3) are satisfied,  $|I(\omega)|\leq 4^{t}$ and $I(\omega) \subseteq supp(\Omega^{<{t}})$ (in place of (A4)) as follows. For a symbol $\omega\in \Omega$ we say that the pair $\{i,j\} \subset [n]$ \emph{inserts a copy of} $\omega$ if 
\begin{itemize}
\item[(C1)] $A_{i,i}, A_{j,j} \in \Omega^{<t}$,
    \item[(C2)]  $\omega\in \{A_{i,j},A_{j,i}\}$, and
     \item[(C3)] the sets $R(i), R(j), R(I(A_{i,i}))$ and $R(I(A_{j,j}))$ are all pairwise disjoint.
\end{itemize}
Next, we greedily construct a maximal set $Y_t$ of pairwise disjoint\footnote{We take a pair of ordered triples to be disjoint if they are disjoint as sets.} triples $(\omega,i,j)$ such that $\{i,j\}$ inserts a copy of $\omega$ which is \emph{not} featured in $\Omega^{<t}$. We define 
\[\Omega_t:=\{\omega: (\omega,i,j)\in Y_t \text{ for some }i,j\in [n]\}.\]
For $\omega\in \Omega_t$ we let $i(\omega), j(\omega)$ be such that $(\omega,i(\omega),j(\omega))\in Y_t$. Also define 
\[I(\omega):=\{i(\omega),j(\omega)\}\cup I(A_{i(\omega),i(\omega)}) \cup I(A_{j(\omega),j(\omega)})\]
and 
\[S_{I(\omega)}:=S_{I(A_{i(\omega),i(\omega)})} \cup S_{I(A_{j(\omega),j(\omega)})} \cup \{(i(\omega),j(\omega)), (j(\omega),i(\omega))\}.\] 
Recall that $I\subseteq  R(I)$ for any index set $I$. Thus, (C3)  implies that the sets $\{i(\omega),j(\omega)\}, I(A_{i(\omega),i(\omega)})$ and $I(A_{j(\omega),j(\omega)})$ are pairwise disjoint, and in extension that $S_{I(\omega)}$
is a permutation of $A_{I(\omega)}$.

The idea behind the above construction is to first concatenate for $A_{i,i}, A_{j,j} \in \Omega^{<t}$ the permutations 
$S_{I(A_{i,i})}$ and $S_{I(A_{j,j})}$ to potentially build a new permutation $S_{I(A_{i,i}) \cup I(A_{j,j})}$ so that the pair $(I(A_{i,i}) \cup I(A_{j,j}), S_{I(A_{i,i}) \cup I(A_{j,j})})$ has the Property (P1). In such a case every symbol that appears on the diagonal of $A$ also appears in $Extend(S_{I(A_{i,i}) \cup I(A_{j,j})})$. Additionally we hope to get the symbols $A_{i,i}$ and $A_{j,j}$ appearing twice in $Extend(S_{I(A_{i,i}) \cup I(A_{j,j})})$ with one of their appearances being on the diagonal of $A$, in the cells corresponding to $A_{i,i}$ and $A_{j,j}$ respectively. In such a case, as the permutation $S_{I(\omega)}$, for some $\omega\in \{A_{i,j},A_{j,i}\}$ can be obtained from 
$Extend(S_{I(A_{i,i}) \cup I(A_{j,j})})$ by exchanging the cells corresponding to  $A_{i,i}$ and $A_{j,j}$ with the cells corresponding to $A_{i,j}$ and $A_{j,i}$ we have that $S_{I(\omega)}$ sees at least as many distinct symbols as $Extend(S_{I(A_{i,i}) \cup I(A_{j,j})})$. Now, the ``additional"  symbol $\omega$ either does not appear in $Extend(S_{I(A_{i,i}) \cup I(A_{j,j})})$, hence $S_{I(\omega)}$ contains an extra symbol, or it appears multiple times in $Extend(S_{I(\omega)})$ and can potentially be used in subsequent steps. 

For the above to succeed there some technicalities that one should be careful with. We deal with these by introducing (A1)-(A4) and (C3). We now prove, in \cref{claim:claim0}, that the elements in $\Omega_t$ satisfy the required properties. 
\begin{claim}\label{claim:claim0}
For $1\leq t\leq \tau$ and $\omega\in \Omega_t$, (A1)-(A2) are satisfied, $|I(\omega)|\leq 4^{t}$ and $I(\omega) \subseteq supp(\Omega^{<{t}})$. In addition, $supp(\omega) \subseteq I(\omega)$ or $\omega$ does not appear in $diag(A)$ (in place of (A3)).
\end{claim}
\begin{claimproof}
Fix $1\leq t\leq \tau$ and $\omega \in \Omega_t$. To ease the notation we write $i,j$ for $i(\omega)$ and $j(\omega)$ respectively. Note that the elements in $\Omega_0$ trivially satisfy (A2) and (A4). 

Condition (C1) implies that $i,j \in supp(\Omega^{<t})$, and $A_{i,i},A_{j,j}$ satisfy (A4). Therefore, \[|I(\omega)| = |\{i,j\}|+| I(A_{i,i})|+| I(A_{j,j}) | \leq 2+ 4^{t-1}+ 4^{t-1}<4^t\] 
and 
\[I(\omega) \subseteq  supp(\Omega^{<t}) \cup  supp(\Omega^{<t-1}) \cup  supp(\Omega^{<t-1}) = supp(\Omega^{<t}).\]

Condition (C2) implies that $S_{I(\omega)}$ contains $\omega$ and so (A1) is satisfied.

As  $\Omega_t$ is disjoint from $\Omega^{<t}$ we have that $\omega  \notin  \Omega^{<t}$. Therefore, if $\omega$ appears in $diag(A)$, then $supp(\omega)\not\subseteq  supp(\Omega^{<t})$. On the other hand, from the above, $I(\omega)\subseteq supp(\Omega^{<t})$. Therefore $supp(\omega)\not\subseteq  I(\omega)$ or $\omega$ does not appear in $diag(A)$.

It remains to prove that (A2) is satisfied, i.e. $(I(\omega),S_{I(\omega)})$ has the property (P1). For this, it suffices to show that for $q\in I(\omega)$ the symbol $A_{q,q}$ either appears in $S_{I(\omega)}$ or on the diagonal of $A_{[n]\setminus I(\omega)}$ (as for $q\notin I(\omega)$, $A_{q,q}$ clearly appears in $diag(A_{[n]\setminus I(\omega)})\subseteq Extend(S_I)$). We do so by considering the following case distinction.  
\vspace{3mm}
\\\textbf{Case 1: $|supp(A_{q,q})| > C$.} 
 Then $A_{q,q}$ appears at least $C+1$ times on the diagonal of $A$. As $|I(\omega)|\leq  4^t \leq 4^{\tau}=C$ we have that there exist a copy of $A_{q,q}$ that appears on the diagonal of $A_{[n]\setminus I(\omega)}$.
 \\ \textbf{Case 2: $q=i, A_{i,i} \in \Omega_0$ (similarly if $q=j, A_{j,j} \in \Omega_0$) and $|supp(A_{q,q})| \leq C$.} As $A_{i,i}$ appears at most $C$ times on $diag(A)$, all of the corresponding cells are indexed by $R(i)$. In addition, since $A_{i,i}\in \Omega_0$ we have that $A_{i,i}$ appears at least twice on the diagonal of $A$, and, by definition, $I(A_{i,i})=\emptyset$. Thereafter, since $I \subseteq R(I)$ for all $I\subseteq [n]$, (C3) gives that among the sets $\{i\}, \{j\}, I(A_{i,i})(=\emptyset)$ and $I(A_{j,j})$ only the set $\{i\}$ intersects $R(i)$. Therefore, $R(i)\setminus I(\omega)$ contains at least one index $r$ and the corresponding cell on the diagonal has entry $A_{r,r}=A_{i,i}$ and belongs to the diagonal of $A_{[n]\setminus I(\omega)}$. 
\\\textbf{Case 3: $q\in \{i,j\}$, $A_{q,q}\notin \Omega_0$ and $|supp(A_{q,q})| \leq C$.}   Then, by (A1), $A_{q,q}$ appears in some cell of $S_{I(A_{q,q})}$. In this case, as $S_{I(\omega)}$ contains all the cells of $S_{I(A_{q,q})}$ (by definition) we have that  $A_{q,q}$ appears in some cell of $S_{I(A_{q,q})}$.
\\\textbf{Case 4: $q\in I(A_{i,i})$ (similarly if $q\in I(A_{j,j})$) and $|supp(A_{q,q})|\leq C$.}  Then, as $(I(A_{i,i}), S_{I(A_{i,i})})$ has the property (P1) we have that either $A_{q,q}$ belongs to $S_{I(A_{i,i})}$ or it appears on the diagonal of $A_{[n]\setminus I(A_{i,i})}$. In the first case as $S_{I(\omega)}$ contains $S_{I(A_{i,i})}$, $S_{I(\omega)}$ also contains $A_{q,q}$.

As $A_{q,q}$ appears at most $C$ times on $diag(A)$, all of the corresponding cells are indexed by $R(q)$. Therefore, in the second case, (C3) implies that no cell on the diagonal of $A_{[n]\setminus I(A_{i,i})}$ with the symbol $A_{q,q}$ (i.e. indexed by $R(q)$) belongs to $A_{I(A_{j,j})}\cup \{i,j\}$, thus at least one such cell belongs to  $diag(A_{[n]\setminus I(\omega)})$.
\end{claimproof}

Thus we have defined $\Omega_t$ so that each $\omega\in \Omega_t$ has an index set $I(\omega)$ and a permutation $S_{I(\omega)}$
satisfying (A1)-(A2), $|I(\omega)|\leq 4^{t}$ and $I(\omega) \subseteq supp(\Omega^{<{t}})$. In addition, if $\omega$ appears in $diag(A)$, then (A3) is also satisfied. Note that this step took time $O(n^2)$ coming from checking all the pairs $\{i,j\}\subset [n]$ to see if they insert some $\omega \notin \Omega^{<t}$. Now if there is some $\omega_*\in \Omega_t$ that does not feature in $diag(A)$, we claim that we are done.  Indeed in such a case, as $(I(\omega_*),S_{I(\omega_*)})$ has the property (P1) and $\omega_*\in  S_{I(\omega)}$ we have that the permutation $Extend(S_{I(\omega_*)})$ has more distinct elements than $diag(A)$ and so we have found a  transversal larger than the one contained in $diag(A)$. 

If $\Omega_t$ does not contain any such $\omega_*$, then  every $\omega\in \Omega_t$ appears in $diag(A)$. In this case  we continue to the next iteration $t+1$. Now the proof of \cref{thm:main} follows from the following claim.
\begin{claim}\label{claim:claim1}
For $1\leq t \leq \tau$, if the largest  transversal contained in $diag(A)$ has size less than $\left(1-\tfrac{\beta}{4}-\eps\right)n$  and $\Omega^{<t}$ contains only elements appearing on the diagonal of $A$  then $|\Omega_t|\geq  \eps \beta  n/12$. 
\end{claim}

Assuming  that the diagonal of $A$ contains less than $(1-\beta/4-\eps)n$ distinct elements and $\Omega^{<\tau}$ contains only elements appearing on the diagonal of $A$, 
then as  the sets $\Omega_1,\Omega_2,..., \Omega_{\tau}$ are pairwise disjoint, \cref{claim:claim1}  implies that $\cup_{t=1}^{\tau} \Omega_t$ contains at least $\tau \cdot \eps \beta  n/12= n$ distinct elements. Therefore there exists $\omega_* \in \cup_{t=1}^{\tau} \Omega_t$ that does not belong to the diagonal of $A$. So after at most $\tau$ iterations, the algorithm will terminate and deliver a larger  transversal. 

\begin{claimproof}[Proof of \cref{claim:claim1}] 
We start by proving the statements \cref{induction:base_case} and \cref{induction} stated below.
\begin{align}
\text{For $r\in [n]$ there exist at most $C$ indices $s\in[n]$  such that $r\in R(s)$.}\label{induction:base_case}    
\end{align}
This follows because, by definition, each of the sets $R(r)$ has size at most $C$. 

\begin{align}
\text{For $t\geq 0$, $r\in [n]$ there are at most $4^tC$ indices $s\in [n]$ with  $A_{s,s}=\omega\in \Omega^{<t+1}$ and  } r \in R(I(\omega)).\label{induction}  
\end{align}

The base case $t=0$ holds trivially as $I(\omega)=\emptyset$ for $\omega\in \Omega_0$ and $R(\emptyset)=\emptyset$. Let $1\leq t \leq \tau$, assume that \eqref{induction} holds for $t-1$ and fix $r\in [n]$. Let $J_1$ be the set of indices $s'$ such that $r$ belongs to $R(s')$ and $J_2$ be the set of indices $s'$ such that $A_{s',s'}\in \Omega^{<t}$ 
and $r$ belongs to $R(I(A_{s',s'}))$. By the definition of $I(\omega)$, we have that
\begin{align*}
R(I(\omega))& =R(i(\omega))\cup R(j(\omega))\cup  R(I(A_{i(\omega),i(\omega)})) \cup  R(I(A_{j(\omega),j(\omega)})).    
\end{align*}    
 Thus $r$ belongs to $R(I(\omega))$ for $\omega\in\Omega_t$ only if $\{i(\omega),j(\omega)\} \cap (J_1\cup J_2)\neq \emptyset$.
 By \cref{induction:base_case} and the induction hypothesis  there exist at most $C+4^{t-1}C$ many such choices 
  for $i(\omega)$ or $j(\omega)$ so that $\{i(\omega),j(\omega)\} \cap (J_1\cup J_2)\neq \emptyset$.   Each such choice appears in at most one triple $(\omega,i(\omega),j(\omega)) \in Y_t$, since the triples in $Y_t$ are pairwise disjoint. Thus there are  at most $2(C+4^{t-1}C)+4^{t-1}C \leq 4^tC$ indices $s\in[n]$ with  $A_{s,s}=\omega \in \Omega_t\cup \Omega^{<t}=\Omega^{<t+1}$ such that  $r\in R(I(\omega))$. This completes the induction and proof of \eqref{induction}. 

Now observe that a pair $(i,j)$ with $A_{i,i}, A_{j,j} \in \Omega^{<t}$ inserts a copy of some element $\omega\in \{A_{i,j},A_{j,i}\}$,  whenever the condition (C3) is satisfied. If (C3) is not satisfied then there exists $\ell \in [n]$ that belongs to at least $2$ of the sets $R(i(\omega)), R(j(\omega)), R(I(A_{i(\omega),i(\omega)}))$ and $R(I(A_{j(\omega),j(\omega)}))$. In such a case we say that $\ell$ witnesses that the pair $(i,j)$ is useless. 
The observations \cref{induction:base_case} and \cref{induction} imply that each $\ell\in [n]$  witnesses the uselessness of at most $\binom{4}{2}4^tC\leq 4^{3t}C$ such pairs.

Let  $n_t=|supp (\Omega^{<t})|$ and $m_t=|\Omega^{<t}|$.  As $diag(A)$ contains fewer than $(1-\beta/4-\eps)n$ distinct symbols and $\Omega^{<t}$ features only symbols in $diag(A)$, we have that $n_t-m_t\geq (\beta/4 +\eps)n$.  The above arguments imply that at step $t$ there exists a multiset $Z_t$ of at least $2\cdot \binom{n_t}{2} -  4^{3t}Cn$ triples $(\omega,i,j)$ such that $\{i,j\}$ inserts $\omega$ and $i<j$. In this multiset $Z_t$ the triple $(\omega,i,j)$ appears twice if $A_{i,j}=A_{j,i}=\omega$ (the two occurrences count for the two distinct occurrences of $\omega$) and once otherwise. Thus each element of $Z_t$ corresponds to a choice of a pair $\{i,j\}\subseteq supp(\Omega^{<t})$ and  a choice of one of the cells corresponding to $A_{i,j}$ or $A_{j,i}$. From those triples at most $\beta m_t n$ contain an element $\omega \in \Omega^{<t}$ as $A$ is $\beta$-bounded. From the rest, each triple shares the same element $\omega$ with at most $\beta n$ other triples (corresponding to the occurrences of $\omega$ in $A$)  and the same $i$ or $j$ with at most $4n$ other triples (corresponding to the cells in the columns and rows indexed by $i$ or $j$). In total, each of the triples intersects at most $5n$ other triples. Thus, due the maximality of $Y_t$,
\begin{align*}
    |\Omega_t|&\geq \frac{n_t(n_t-1) - 4^{3t}C n -\beta m_t n}{5n}
     =\frac{n_t^2-\beta m_tn}{5n}- \frac{n_t}{5n} -\frac{ 4^{3t}C}{5}
     \geq \frac{\left(m_t+ \left(\frac\beta4 +\eps\right)n\right)^2-\beta m_tn}{5n}-4^{3t}C
    \\ &= \frac{ \left(\frac\beta4 n -m_t\right)^2+\eps \beta n^2/2+\eps^2n^2+2\eps m_t n}{5n}-4^{3t}C \geq \frac{\eps \beta n^2}{10n}-4^{3t}C \geq \frac{\eps \beta n}{12},
\end{align*}
using here that $n$ is sufficiently large. This concludes the proof of the claim and hence the theorem.
\end{claimproof}
\end{proof}

\bibliographystyle{amsplain_initials_nobysame_nomr}
\bibliography{main.bib}

\providecommand{\bysame}{\leavevmode\hbox to3em{\hrulefill}\thinspace}
\providecommand{\MR}{\relax\ifhmode\unskip\space\fi MR }
\providecommand{\MRhref}[2]{%
  \href{http://www.ams.org/mathscinet-getitem?mr=#1}{#2}
}
\providecommand{\href}[2]{#2}
\begin{thebibliography}{10}

\bibitem{aharoni2017conjecture}
R.~Aharoni, E.~Berger, D.~Kotlar, and R.~Ziv, \emph{On a conjecture of {S}tein}, Abh. {M}ath. {S}emin. {U}niv. {H}ambg. \textbf{87} (2017), 203--211.

\bibitem{bissacot2011improvement}
R.~Bissacot, R.~Fern{\'a}ndez, A.~Procacci, and B.~Scoppola, \emph{An improvement of the {L}ov{\'a}sz local lemma via cluster expansion}, Combin. Probab. Comput. \textbf{20} (2011), no.~5, 709--719.

\bibitem{brualdi1991combinatorial}
R.~A. Brualdi and H.~J. Ryser, \emph{Combinatorial matrix theory}, Cambridge University Press, 1991.

\bibitem{ugh}
D.~Chakraborti, M.~Christoph, Z.~Hunter, R.~Montgomery, and T.~Petrov, \emph{Almost-full transversals in equi-$n$-squares}, Personal communication.

\bibitem{erdos1991lopsided}
P.~Erd{\H{o}}s and J.~Spencer, \emph{Lopsided {L}ov{\'a}sz local lemma and {L}atin transversals}, Discrete Applied Math. \textbf{30} (1991), no.~151-154, 10--1016.

\bibitem{harris2020new}
D.~G. Harris, \emph{New bounds for the {M}oser-{T}ardos distribution}, Random Struct. Algorithms \textbf{57} (2020), no.~1, 97--131.

\bibitem{harris2017algorithmic}
D.~G. Harris and A.~Srinivasan, \emph{Algorithmic and enumerative aspects of the {M}oser-{T}ardos distribution}, ACM Trans. Algorithms \textbf{13} (2017), no.~3, 1--40.

\bibitem{montgomery2023proof}
R.~Montgomery, \emph{A proof of the {R}yser-{B}rualdi-{S}tein conjecture for large even $ n$}, arXiv preprint arXiv:2310.19779 (2023).

\bibitem{moser2010constructive}
R.~A. Moser and G.~Tardos, \emph{A constructive proof of the general {L}ov{\'a}sz local lemma}, Journal of the ACM (JACM) \textbf{57} (2010), no.~2, 1--15.

\bibitem{pokrovskiy2019counterexample}
A.~Pokrovskiy and B.~Sudakov, \emph{A counterexample to {S}tein’s {E}qui-$n$-square conjecture}, Proc. Am. Math. Soc. \textbf{147} (2019), no.~6, 2281--2287.

\bibitem{ryser1967neuere}
H.~J. Ryser, \emph{Neuere probleme der kombinatorik}, Vortr{\"a}ge {\"u}ber Kombinatorik, Oberwolfach (1967), 69--91.

\bibitem{stein1975transversals}
S.~K. Stein, \emph{Transversals of {L}atin squares and their generalizations.}, Pacific J. Math. \textbf{59} (1975), 567--575.

\end{thebibliography}

\end{document}